\documentclass{IEEEtran}
\usepackage{romannum}
\usepackage[usenames,dvipsnames]{xcolor}
\usepackage{tkz-berge}
\usetikzlibrary{fit,shapes}                                     
\usepackage{calrsfs}
\usepackage{graphicx}
\usepackage{graphics} 
\usepackage{epsfig}    
\usepackage{mathptmx}
\usepackage{caption}
\usepackage{subcaption}
\usepackage{times} 
\usepackage{tikz}
\usetikzlibrary{positioning}
\usepackage{amsmath, amssymb}
\usetikzlibrary{arrows}
\renewcommand{\baselinestretch}{1.02}

 \usepackage[T1]{fontenc}

\usepackage{kantlipsum}
\usepackage{amsthm}
\usepackage{amsfonts} 
\usepackage{setspace}
\usepackage{multirow}
\usepackage{textcomp}
\usepackage[english]{babel}
\usepackage{float} 
\usepackage{algorithmicx}
\usepackage[section]{algorithm}
\usepackage{algpascal}
\usepackage{algc}
\usepackage{algcompatible}
\usepackage{algpseudocode}
\let\oldReturn\Return
\renewcommand{\Return}{\State\oldReturn}
\usepackage{mathtools}
\usepackage{bm}
\usepackage{mathptmx}
\usepackage{times} 
\usepackage{mathtools, cuted}
\usepackage{textcomp}
\usepackage[english]{babel}
\usepackage{rotating}
\usepackage{pgfplots}
\pgfplotsset{compat=1.5}
\usepackage{siunitx}
\usepackage{pgfplotstable}
\usepackage{filecontents}

\def\BibTeX{{\rm B\kern-.05em{\sc i\kern-.025em b}\kern-.08em
    T\kern-.1667em\lower.7ex\hbox{E}\kern-.125emX}}

\newtheorem{theorem}{Theorem}
\newtheorem{corollary}[theorem]{Corollary}
\newtheorem{lemma}[theorem]{Lemma}
\newtheorem{prop}{Proposition}
\newtheorem{prob}{Problem}
\newtheorem{defi}{Definition}

\newcommand{\R}{\mathbb{R}}
\newcommand{\A}{\mathbf{A}}
\newcommand{\B}{\mathbf{B}}
\newcommand{\AG}{\mathbf{A}(G)}
\newcommand{\BS}{\mathbf{B}(S)}
\newcommand{\tA}{{\bf A}_{\times}}
\def \*{\star}
\def \proba{\mathbb{P} }

                     \makeatletter
\makeatletter
\newcommand{\specificthanks}[1]{\@fnsymbol{#1}}
\makeatother
\title{Minimizing Inputs for Strong Structural Controllability} 
\author{Kumar Yashashwi\thanks{\textsuperscript{$^*$}Department of Electrical Engineering, Indian Institute of Technology Bombay, India. Email: $\lbrace$kryashashwi,chaporkar$\rbrace$@ee.iitb.ac.in.}\textsuperscript{\specificthanks{1}}, Shana Moothedath\thanks{\textdagger Department of Electrical and Computer Engineering,  University of Washington, USA. Email: sm15@uw.edu}\textsuperscript{\specificthanks{2}}, and Prasanna Chaporkar\textsuperscript{\specificthanks{1}}
}

\begin{document}
\maketitle

\begin{abstract}
	The notion of strong structural controllability ({\it s-controllability}) allows for determining controllability properties of 
	large linear time-invariant systems even when numerical values of the system parameters are not known {\em a priori}. 
	The s-controllability guarantees controllability for all numerical realizations of the system parameters. We address the optimization  problem
	of minimal cardinality input selection for s-controllability. 
	Previous work shows that not only the optimization problem is NP-hard, but  finding an approximate solution is also hard. 
	We propose a randomized algorithm using the notion of zero forcing sets to obtain an optimal solution with high probability.
	We compare the performance of the proposed algorithm with a known heuristic \cite{ChaMes:13} 
	for synthetic random systems and five real world networks, viz. IEEE 39-bus system, re-tweet network, protein-protein interaction network, US airport network, and a network of physicians. It is found that our algorithm performs much better than the 
	heuristic in each of these cases. 
\end{abstract}
\section{Introduction}\label{intro}
With the ever increasing interconnection of networks in today's world, study of complex networks has gained a lot of attention. Large interconnected networks are prevalent in several fields including power systems \cite{PeqPopIliAgu-13}, the Internet \cite{LiuSloBar:11}, biology \cite{Guetal:15},  and multi-agents \cite{ZamLin:09}. Controllability of complex systems is an essential property to achieve intended performance of the system.  
Analyzing controllability of large-scale systems using conventional control theory  is a challenging task as the exact 
numerical entries of the system matrices are often unknown.  {\em  Structural analysis} is a framework to analyze complex systems using the sparsity pattern of their network dynamics \cite{LiuBar:16}. 

{\em Weak structural controllability} and  {\em strong structural controllability} are  key concepts in structural analysis to study controllability of the networks. 
While {\em weak} structural controllability guarantees controllability of \textit{almost all} numerical realizations of the network there 
may exist a thin set\footnote{A non-trivial algebraic variety is `thin' and a set of measure zero \cite{Rei:88}.} of uncontrollable systems with the same sparsity pattern. Weak structural controllability thereby does not guarantee Kalman's controllability \cite{Kai:80}. This has led to the notion of  strong  structural controllability (s-controllability) \cite{MayYam-79}. The {\em s-controllability} of a system guarantees controllability of {\em all} possible numerical systems whose sparsity pattern is same as  that of the network  \cite{JarSvaAlt:11}. s-controllability is a graph theoretic analysis which is useful in two contexts: 
(i)~it verifies controllability of all numerical systems with the given sparsity pattern when only the
sparsity patterns of the state matrices are known, and 
(ii)~it verifies  controllability of systems with known  numerical entries bypassing the  {\em numerical instability} involved in the floating point computations \cite{LiuBar:16}. 

Optimal input design for s-controllability is an important optimization problem in complex systems due to the large system size. 
Actuating many states incurs huge cost in terms of installation and monitoring.  
There has been some effort on input design for  s-control \cite{ChaMes:13},  \cite{PeqPopIliAgu-13},  \cite{TreDel-15}. 
However, they do not provide a computationally efficient scheme/algorithm with guaranteed performance ratio for sparsest input matrix design  on general graph topology.
In this paper, we propose a randomized algorithm based on Markov Chain Monte Carlo (MCMC)  using {\em zero forcing set} (ZFS) \cite{TreDel-15}  to design a sparsest input matrix that guarantees s-controllability. 
In our formulation,   we construct a discrete time Markov chain (DTMC) and a \textit{cost function}, based on the color change rule of ZFS given in \cite{TreDel-15}, and show that the MCMC returns an optimal solution asymptotically.  
  We make the following contributions:\

\noindent
$\bullet$ We propose a computationally efficient randomized algorithm to obtain minimum input selection for s-controllability,
which is known to be notoriously hard even to approximate. 
The proposed algorithm exploits known equivalence between ZFS and s-controllability for quick convergence.

\noindent
$\bullet$  We validate our proposed algorithm on small networks and on large self-damped directed trees for which optimal can
be computed. We show that the algorithm obtains optimal solution in all of the considered cases. We also compare the performance
of our algorithm with that of only known heuristic \cite{ChaMes:13} for large randomly generated networks. 
The experiments show that our algorithm performs much better.

\noindent
$\bullet$  
We  also conduct experiments on five real-world networks, viz. IEEE 39-bus system, re-tweet network,  protein-protein interaction network, US airport network, and network of physicians. In these real world scenarios as well, our algorithm significantly outperforms the heuristic.

\noindent
$\bullet$ Our scheme can also be used for efficient computation of ZFS in a graph. As computation of ZFS is a key step in
solving many problems including that in quantum control theory \cite{bur:13}, PMU placement problems \cite{nuq:05} and inverse eigenvalue problems \cite{bar:08}.
Hence, our approach can also be used for these problems and hence can be of independent interest.

\subsection{Related Work}
Optimal input selection for weak controllability is addressed in many papers (see  \cite{Ols:15}, \cite{PeqKarAgu_2:16},  \cite{MooChaBel:18} and references therein). A linear time algorithm is given in \cite{weber2014linear} for verifying s-controllability. Paper \cite{menara2017number} deals with the number of control
inputs and algebraic characterization of zero-forcing sets
in the framework of s-controllability. s-controllability is  studied in the context of {\em target controllability} in \cite{WaaCamTre:17}, \cite{MonCamTre:15}, where the focus is on controlling a prescribed subset of states rather than the whole system. Design of an s-controllable network, when only nodes are specified, using the notion of {\em balancing sets} is done in \cite{MouHaeMes-18}. 
Our focus in this paper is on sparsest input design for s-controllability of the full network. Further, we consider s-controllability of linear time-invariant systems differing from \cite{GraGarKib-17} and \cite{HarReiSva-13} which consider linear time {\em varying} systems. We now describe the most relevant work to this paper.

Paper \cite{ChaMes:13} proved NP-hardness of sparsest input design problem for s-controllability (using {\em constrained t-matching}) and proposed a heuristic algorithm. The performance of the heuristic algorithm is analyzed by comparing the results with optimal results obtained using exhaustive search method, for smaller networks.  We construct problem instances  to demonstrate that the heuristic algorithm in \cite{ChaMes:13} gives arbitrarily bad performance ratio for which our MCMC-based algorithm gives an optimal solution. Moreover, in the several experiments we conducted, the proposed MCMC method always outperforms the heuristic.  Performance
of our scheme is found to be much better than the heuristic when the state matrix is dense.
Paper \cite{PeqPopIliAgu-13} gave an algorithm to find a sparsest input matrix for s-control when the structured state matrix of the system has so-called {\em maximal staircase} structure. 
On the other hand,  \cite{TreDel-15} proved the equivalence  of s-controllability and  ZFS problem and using the equivalence gave a polynomial-time algorithm to find sparsest input matrix for s-controllability in {\em self-damped} systems with a tree-structure. 
We  do not assume any special graph topology or matrix structure. We give a computationally efficient algorithm to design a sparsest input matrix for s-control of systems with general graph topology. 

\subsection{Organization}\label{subsec:org}
The rest of the paper is organized as follows: Section~\ref{sec:ssc} discusses notations and preliminaries on s-controllability and describes the problem addressed in this paper. Section~\ref{sec:prelim} describes the zero forcing problem and its relation to s-controllability. Section~\ref{sec:sa} presents an MCMC based approach to minimize the cardinality of input set required for s-controllability.    Section~\ref{sec:results} demonstrates the performance of MCMC on synthetic as well as real-world networks. Finally, section~\ref{sec:conclu} gives the concluding remarks on our results.
\section{Strong structural controllability and problem formulation}\label{sec:ssc}
Consider a structured Linear Time-Invariant system (LTI) given by $\dot{x} = {\bf A} x + {\bf B}u$, where $x \in \R^n$ is the system state,
${\bf A} \in \{0, \times \}^{n \times n}$ is the state matrix and ${\bf B} \in \{0, \times \}^{n \times m}$ is the input matrix. Here $\times$ denotes independent free parameters. The LTI system is called
{\em structured system} as exact numerical values in ${\bf A}$ and ${\bf B}$ are not known and only the sparsity pattern is known. The structured system
represents a class of LTI systems $\dot{x} = \tilde{A} x + \tilde{B} u$ such that numerical entry $\tilde{A}_{ij} \not= 0$ 
($\tilde{B}_{ij} \not= 0$, resp.) if and only if
${\bf A}_{ij} =  \times$ (${\bf B}_{ij} =  \times$, resp.). Any such $(\tilde{A},\tilde{B})$ is called {\it numerical realization}  of the structured system
$({\bf A},{\bf B})$.

\begin{defi}
	A structured LTI system $({\bf A},{\bf B})$ is called s-controllable if it is controllable for all of its numerical realizations $(\tilde{A},\tilde{B})$. 
\end{defi}

In this paper, we assume that the structural state matrix ${\bf A}$ is given and our aim is to design a sparsest input matrix ${\bf B}$
such that the LTI system $({\bf A},{\bf B})$ is s-controllable. 
Now, we formally specify our problem. 
Towards this end, let us define ${\cal K}_A$ to be the set of all $n \times m$ structured input matrices 
${\bf B}$ such that $({\bf A},{\bf B})$ is s-controllable and $m$ can take any value in $\{1,\ldots,n\}$. 
Clearly ${\cal K}_A$ is not empty as the $n \times n$ matrix $I$ with
$I_{ii} =  \times$ for each $i$ is in ${\cal K}_A$ (see Proposition~\ref{prop:ssc_thm}).
Define, for a structured matrix $Z$, $|Z|$ to be the number of
$ \times$ entries in $Z$.

\begin{prob}\label{prob:ssc}
	Given a structured system with state matrix $\A$, find
	\[ \B^{\*} ~\in~ \arg\min_{\B \in {\cal K}_A } |\B|. \] 
\end{prob}

{\it Remark: } 
In \cite{ChaMes:13}, it is shown that to obtain $\B^{\*}$, it is enough to consider $n\times m$ $\B$ matrices
that have exactly one non-zero entry in each column and each row. Thus, Problem~\ref{prob:ssc} can be alternatively stated as follows: find a smallest cardinality set $S^\*$ of system states such that actuating them independently
results in a s-controllable system.

In the next section, we provide preliminaries and state relevant known results.

\section{Preliminaries} \label{sec:prelim}
For a given structured LTI system $(\A,\B)$, Reinschke {\it et al.} have given a necessary and sufficient condition to verify
its s-controllability in their seminal paper \cite{ReiSvaWen-92}. We introduce the following concepts before stating
the main result from \cite{ReiSvaWen-92} for the sake of completeness.

\begin{defi}
	A structured pair $(\A, \B)$ is said to be in Form~III if there exists two permutation matrices $P_1$ and $P_2$ such that 
	\[
	P_1\,[\A~ \B]\,P_2=
	\begin{bmatrix}
	\otimes & \cdots & \otimes & \times & 0 & \cdots & 0  \\
	\vdots &  & \vdots & \ddots & \times &  \ddots & \vdots\\
	& & & & \ddots & \ddots & 0\\
	\otimes &  \cdots & \otimes & \cdots & \cdots & \otimes & \times\\
	\end{bmatrix},
	\] 
	where the $\times$ denotes the location of nonzero elements and $\otimes$ can be either zero or nonzero. 
\end{defi}

\begin{defi}
	A modified pattern $\A_\times$ of matrix $\A$ is obtained by replacing the zero entries
	on the diagonal of $\A$ by $\times$.
\end{defi}

Next we state the key result of \cite{ReiSvaWen-92}.

\begin{prop}[Theorem~2, \cite{ReiSvaWen-92}]\label{prop:ssc_thm}
	The structured pair $(\A, \B)$ is s-controllable if and only if
	\begin{enumerate}
		\item the pair $(\mathbf{A,B})$ is of Form~III, and
		\item  the matrix $(\mathbf{A_\times,B})$ can be transformed into Form~III such that the $\times$ elements do not correspond to the diagonal elements in $\mathbf{A}$ that were $\times$'s.
	\end{enumerate}
\end{prop}

The verification of s-controllability of pair $(\A,\B)$ can be done in polynomial time by verifying two conditions stated in Proposition~\ref{prop:ssc_thm} \cite{ReiSvaWen-92}. Recently, Chapman {\it et al.} \cite{ChaMes:13} and Trefois {\it et al.} \cite{TreDel-15} have given graph theoretic algorithms based on the {\it constrained $t$-matching} and 
{\it zero forcing problem}, respectively, for verification of s-controllability when the inputs are dedicated.
Note that for dedicated inputs, the input matrix $\B$ can be obtained as the matrix $\left[e_{i_1}, \ldots, e_{i_m}    \right] \in \{0,\times\}^{n \times m}$.  Here, $e_i$ denotes the structural form of the $i^{\rm th}$ standard basis vector, i.e., $[e_{i}]_{j} = \times$ for $j=i$ and $[e_{i}]_{j} = 0$ otherwise. Since, our work relies on the 
conditions
based on zero forcing problem, we describe that in detail next.

\subsection{Zero Forcing Problem in Graphs} \label{sec:zfs}
The zero forcing problem in a directed graph $G=(V, E)$ is  introduced in \cite{BarBarButCioCveFalGodHaeHogMik-08}.
Nodes in the graph with some abuse of notation are denoted by $x_i$ with $i = 1,\ldots,n$. The edge set $E \subseteq V \times V$.
A node $x_i \in V$ is an out-neighbor of node $x_j \in V$, if there exists a directed edge $(x_j, x_i) \in E$.
Note that if $(x_i,x_i) \in E$ then $x_i$ is also an out-neighbor of $x_i$.
Now, the zero forcing problem can be broadly described as follows. 
Each node in the graph can either be black or white. Let $S \subseteq V$ be the set of black nodes at the beginning.
Note that $S$ can be the empty set.
White nodes can be forced to change their colour to black based on some
{\it colour changing rule} applied iteratively until no colour change as per the specified colour changing rule is possible. If starting from set $S$, all nodes eventually become black, then
$S$ is called {\it zero forcing set} for graph $G$. Next we define  graphs of interest to us and their corresponding colour changing rule 
 \cite{BarBarButCioCveFalGodHaeHogMik-08}. 

\begin{defi}[Loop Directed Graph]\label{defi:ldg}
	A directed graph $G=(V,E)$ is called a {\it loop directed} if it allows self loops, i.e. $(x_i,x_i)$ can belong to $E$ for any node $x_i\in V$.
\end{defi}


\begin{defi}[Colour Changing Rule]\label{defi:ccr}
	If exactly one out-neighbor $x_j$ of $x_i$ is white, then change color of $x_j$ to black. 
\end{defi} 

{\it Remark :} Note that colour of $x_i$ plays no role in deciding whether its out-neighbor $x_j$ changes colour to black unless
$(x_i,x_i)\in E$, i.e. $x_i$ is also an out-neighbor of $x_i$. By Definition~\ref{defi:ccr}, if $(x_i,x_i)\in E$ and all other out-neighbors of $x_i$ are black,
then $x_i$ becomes black.




If $x_j$ changes colour because of $x_i$, then we say that 
$x_i$ \textit{forces} $x_j$ and is denoted by $x_i \rightarrow x_j$. Thus, the subsequent colour changes can be alternatively 
represented in terms of the {\it chronological list of forces} applied. When no vertex can be forced to change colour, 
then the iterative colour change 
process stops. Thus, starting from set $S$ of nodes which are assigned black colour in the beginning, we end up with
set $W(S)$ of white nodes that can not be forced to change colour.

\begin{defi}[Zero Forcing Set] \label{defi:zfs}
	A set $S \subseteq V$ is said to be a zero forcing set for the graph $G$ if $W(S)$ is the empty set, i.e. at the termination all nodes in
	the graph become black.
\end{defi}
 
 \begin{defi}[Zero Forcing Number]\label{defi:zfn}
 	Zero forcing number $Z(G)$ of graph $G$ is defined as \[ Z(G) = \min\{|S| \, : \, S\subseteq V \mbox{\ is a zero forcing set of } G \}.  \]
 \end{defi}
 
Next we elaborate on the relation between zero forcing problem and  s-controllability.
 
 \subsection{Zero Forcing and Strong Structural Controllability}
Corresponding to the state matrix $\A$ of LTI system, we can define directed graph $G(\A) = (V,E)$ as follows. 
Define $V$ equals set of states and $(x_j,x_i) \in E$ if $\A_{ij} =  \times$. Also, corresponding to any $S \subseteq V$ 
such that $|S| = m$,
define $\B(S) \in \{0, \times \}^{n\times m}$ as $\B(S) =  [e_i]_{\{i:x_i \in S\}}$.
Thus, $B(S)$ corresponds to the input matrix such that state $x_i$ receives input only if $x_i\in S$.
In terms of the zero forcing problem, the following result gives a necessary and sufficient condition for s-controllability of the  system $(\A, \B)$.
\begin{prop}[Theorem~5.5, \cite{TreDel-15}]\label{prop:ssc_thm_zf}
	The system $(\A, \B)$, where $\B = \BS$, is s-controllable if and only if
	\begin{enumerate}
		\item $S$ is a zero forcing set of $G(\A)$ and
		\item S is a zero forcing set of $G(\A_\times)$ for which there is a chronological list of forces that does not contain any force of the form $x_i \rightarrow x_i$ for any $x_i$ such that $\A_{ii} = \times$.  
	\end{enumerate}
\end{prop}

An example to illustrate the above concepts is provided in Appendix~\ref{app:illustrative_ex}.
The above proposition establishes equivalence between the s-controllability and the 
zero forcing problem. Thus, finding $\B^\*$ in Problem~\ref{prob:ssc} is equivalent to finding a least cardinality
set $S^\*$ that satisfies the two conditions in Proposition~\ref{prop:ssc_thm_zf}. 
Unfortunately, not only finding
a least cardinality ZFS of a graph is NP-hard \cite{TreDel-15, Aaz:08}, but also 
finding an approximate solution is hard \cite{Mis:11}. 
Thus, a heuristic algorithm to solve Problem~\ref{prob:ssc} can be arbitrarily bad. Indeed, we 
show that the heuristic algorithm proposed in \cite{ChaMes:13} can obtain solution which in $O(n)$ suboptimal,
where $n$ is the number of states (see \cite{arviv_version_SSC}).
Hence, we propose a randomized algorithm that returns an optimal solution with high probability. 
Next, we describe our proposed algorithm.


\section{MCMC based Approach}\label{sec:sa}
As stated in the previous section, no polynomial time algorithm exists for determining sparsest input  matrix that can make the system s-controllable  unless $P=NP$.
So, we propose a randomized scheme that outputs optimal solution with high probability. 
Broadly, our approach is as follows: we define a cost function $C(\cdot)$ 
that assigns a non-negative real number to
every subset of $V$. Thus, $C : {\cal V} \rightarrow \Re_+$, where ${\cal V}$ is the power set of $V$
and $\Re_+$ is the set of non-negative real numbers. 
The function achieves the minimum value for a least cardinality input set (say $S^\*$) which 
guarantees s-controllability of the system. 
Next, we design an algorithm that draws sample, say $Y$, from ${\cal V}$ such that for any $S \in {\cal V}$
\begin{equation}\label{eq:sta_prob}
P_T(S) = \proba(Y = S) = \dfrac{e^{-C(S)/T} }{\sum_{S_1\in{\cal V}} e^{-C(S_1)/T} }, 
\end{equation}
where $T \ge 0$ is referred to as temperature parameter. 
Note that as $T \to 0$, $P_T(S)$ converges to 1 if $C(S)$ is a global minimum value of the cost function $C$.
Thus, for a small enough $T$, the algorithm outputs optimal $S$ with high probability. 
Note that computing distribution $[P_T(S)]_{S \in {\cal V} }$ has exponential complexity as 
$|{\cal V}| = 2^{|V|}$. Hence, we propose to use the well known Markov Chain Monte Carlo (MCMC) method for sampling from the desired distribution \cite{bre:13}. A key step in this approach is to construct a
Discrete Time Markov Chain (DTMC) on space ${\cal V}$ whose stationary distribution $\pi_T(S)$ 
point-wise equals the desired distribution $P_T(S)$. Let $\{Y_T(t)\}_{t\ge 0}$ be the DTMC.
Then it is guaranteed that 
\begin{align} \label{eq:dtmc_conv} 
\lim_{t\to\infty}  \proba(Y_T(t) = S ) = \pi_T(S)
\end{align}
for every $S\in{\cal V}$.
Thus, a random variable $Y_T(t)$ has approximately the desired distribution for $t$ large enough.
Hence, to sample the distribution $[P_T(S)]_{S \in {\cal V} }$, we just need to simulate the DTMC for
sufficiently large $t$ and output its state at time step $t$.
For this approach to be feasible, we should be able to execute each transition of the DTMC in polynomial
time and storage complexity. Also, the convergence in Eq.~(\ref{eq:dtmc_conv}) should be fast enough, so that we need
not simulate the DTMC for prohibitively excessive number of steps to get a good approximation.
Choice of the cost function plays a key role in determining the speed of convergence.
Next, we describe the cost function considered in this paper.
 
\subsection{Cost Function}
Consider $S \in {\cal V}$. We now describe how the cost function at $S$, $C(S)$, is calculated.
Recall that $G(\A)$ and $G(\A_\times)$ are the loop directed graphs corresponding to the state matrix $\A$
and the modified state matrix $\A_{\times}$, respectively. We assign black colour to each of the vertices
in $S$. Starting from these black vertices, we apply colour changing rule iteratively in both $G(\A)$ and $G(\A_\times)$. It should be noted that in $G(\A)$ all nodes with self loop can potentially force themselves,
while in $G(\A_\times)$ only the nodes that do not have self loop in $G(\A)$ can force themselves.
Let $W(S)$ and $W_\times(S)$ denote the set of white nodes that can not be forced in $G(\A)$ and $G(\A_{\times})$,
respectively. Now, we define our cost function as follows.
\begin{align}
C(S)&= |S| + (1+\epsilon)|W(S) \cup W_{\times}(S)|,     \label{eq:costfn}
\end{align}
where $\epsilon > 0$ is a constant (typically small). Note that the cost function can be
evaluated for any $S$ in time $O(n^3)$ \cite{TreDel-15}. Now, we show that the cost function achieves 
minimum value at $S^\*$.

\begin{lemma}\label{lemma:cost_opt}
	Let $S^\*$ be such that $\B(S^\*)$ is an optimal solution of Problem~\ref{prob:ssc}, 
	i.e., structured system $(\A,\B(S^\*))$ is s-controllable, and for any $S$ such that $(\A,\B(S))$ is s-controllable, $|S^\*| \leq |S|$. Then, for every $S\in{\cal V}$, $C(S^\*) \leq C(S)$ with equality 
	if and only if $\B(S)$ is also an optimal solution to Problem~\ref{prob:ssc}.
\end{lemma}

\begin{proof}
	We prove the required by contradiction. Let $\tilde{S}$ be a set of states such that 
	$\tilde{S} \in \arg\min_{S\in{\cal V}} C(S)$. We consider two possibilities, viz. 
	\begin{enumerate}
		\item $(\A,\B(\tilde{S}))$ is \textit{not} s-controllable, and
		\item $(\A,\B(\tilde{S}))$ is s-controllable, but $|S^\*| < |\tilde{S}|$.
	\end{enumerate}

In the first case, $|W(\tilde{S}) \cup W_\times(\tilde{S})| \ge 1$ by Proposition~\ref{prop:ssc_thm_zf}.
Define, $S = \tilde{S} \cup W(\tilde{S}) \cup W_\times(\tilde{S})$. $S$ consists of  all the nodes that can not be
forced black starting from $\tilde{S}$  in addition to the nodes in $\tilde{S}$. Now, $S$ satisfies both the conditions in Proposition~\ref{prop:ssc_thm_zf}. Thus,
$(\A,\B(S))$ is s-controllable. Moreover, $C(S) < C(\tilde{S})$ as $\epsilon > 0$
(see Eq.~\eqref{eq:costfn}). Thus, all local minima of the cost function correspond to  feasible solutions of Problem~\ref{prob:ssc}.
	
In case two, since both sets of inputs $S^\*$ and $\tilde{S}$ guarantee s-controllability,
$|W(S) \cup W_\times(S)| = 0$ for $S \in \{\tilde{S}, S^\*\}$ by Proposition~\ref{prop:ssc_thm_zf}. 
Thus, the value of the cost functions at these sets are simply their cardinality by (\ref{eq:costfn}). 
Hence, $C(S^\*) < C(\tilde{S})$. This proves the required.
\end{proof}
In the next section, we discuss construction of DTMC and show that it has the desired stationary distribution.

\subsection{DTMC construction} \label{DTMC}
Let $\{Y_T(t)\}_{t\ge 0}$ be a DTMC on ${\cal V}$, i.e., $Y_T(t) \in {\cal V}$ for every step~$t$. Here,
$T >0$ is a fixed parameter. Note that the DTMC is a set valued process.
To completely specify the DTMC, it is enough to provide the $|{\cal V}| \times |{\cal V}|$
one step transition probability matrix (TPM), which
we do next. First we describe how transition in the DTMC occur.
Fix $S \in {\cal V}\setminus \{\phi,V  \}$ and define the following sets.
\begin{align*}
{\cal N}_1(S)  & = \{S' : S' = S \cup\{x\} \mbox{ for some $x\in V\setminus S$} \}, \\
{\cal N}_2(S)  & = \{S'  : S' = S \setminus\{x\} \mbox{ for some $x\in S$} \}, \\
{\cal N}_3(S)  & = \{S'  : S' = (S \cup\{x\})\setminus\{x'\}  \mbox{ for $x\in V\setminus S$ and $x'\in S$} \}, \\
{\cal N}(S) &= {\cal N}_1(S) \cup {\cal N}_2(S) \cup {\cal N}_3(S) \cup \{S\}.
\end{align*}
Note that the sets in the collection ${\cal N}_1(S)$ are obtained by adding a new vertex to the set $S$.
The collection ${\cal N}_2(S)$ is obtained by removing an existing vertex from set $S$.
Finally, the collection ${\cal N}_3(S)$ is obtained by removing one existing vertex and adding a new vertex
to $S$. 
Note that for $S = \phi$, ${\cal N}_2(S) \cup {\cal N}_3(S) = \phi$ and for $S =V$, 
${\cal N}_1(S) \cup {\cal N}_3(S) = \phi$.
Moreover, $|{\cal N}_1(S)| = n - |S|$, $|{\cal N}_2(S)| = |S|$ and $|{\cal N}_3(S)| = |S|\times (n - |S|)$.
In our proposed DTMC, transitions are possible only to sets in collection ${\cal N}(S)$ from $S$.
In this sense, the set ${\cal N}(S)$ is referred to as the set of neighbors of $S$. 
Let $Y_T(t) = S$. The transition happens
in two stages: first, a neighbor, say $S_p$, of current set $S$ is proposed as possible next state. Second, the
proposed transition is either accepted and $Y_T(t+1) = S_p$ or it is rejected and $Y_T(t+1) = S$.
The acceptance/rejection of proposed transition is probabilistic and the probability depends on the
value of cost function at $S$ and $S_p$. Now, we describe the TPM of the proposed DTMC.
Define for  a set $S$,

\begin{align*}
p_1(S) = \frac{2 (n -|S|)}{3n}, p_2(S) = \frac{2| S |}{3n}, \mbox{  and }  p_3(S) &= \frac{1}{3}.
\end{align*}

If $Y_T(t+1) = S$, the proposed state $S_p$ is picked uniformly at random from ${\cal N}_w(S)$ 
with probability $p_w(S)$ for $w=1,2,3$. Alternate way to think about this choice is the following.
First choose one of the three actions, viz. add a vertex, remove a vertex and swap vertices, with probabilities
$p_1(S)$, $p_2(S)$ and $p_3(S)$, respectively. In case the add action is chosen, pick a vertex from 
$V\setminus S$ uniformly at random and add it to $S$ to get $S_p$.
In case the remove action is chosen, pick a vertex from 
$S$ uniformly at random and remove it from $S$ to get $S_p$. For the swap action, select one vertex each
from $S$ and $V\setminus S$ uniformly at random and swap them to get $S_p$.
Finally, if $S\in \{\phi,V \}$ and the swap action is chosen, then $S_p$ is chosen to be $S$ itself.
Now, when $Y(t) =S$ and a state $S_p \in {\cal N}(S)$ is proposed, then the probability of the next set chosen 
is given by
\begin{align*}
 Y_T(t+1) & = \begin{cases} S_p \mbox{ w.p. } \min \{e^{(C(S)-C(S_p))/T},1\}, \\
 S, \mbox{ otherwise.} \end{cases}
\end{align*}

Here, w.p. stands for with probability. Consider set $S$ and $S' \in {\cal N}_w(S)$ for some $w \in \{1,2,3\}$. Then, for every $t \ge 0$, $P_T({S,S'})  = \proba(Y_T(t+1) = S' | Y_T(t) = S)=$
\begin{equation}
\begin{cases}
\begin{array}{ll}
 \frac{2}{3n} \min \{e^{(C(S)-C(S'))/T},1\}, &\mbox{ if $w=1,2$}, \\
 \frac{1}{3 |S| (n-|S|)}\min \{e^{(C(S)-C(S'))/T},1\},& \mbox{ if $w=3$}. \label{eq:trans1}
 \end{array}
\end{cases}
\end{equation}

Finally, 
\begin{align}
P_T({S,S}) = 1 - \sum_{S' \in {\cal N}(S) \setminus \{S\}} P_T({S,S'}). \label{eq:trans3}
\end{align}

Eqs.~(\ref{eq:trans1}) and~(\ref{eq:trans3}) provide complete description of the TPM of the proposed DTMC.
Next, we establish the desired properties of the proposed DTMC.

\begin{lemma}
	The constructed DTMC $\{{Y_{T}(t)}\}_{t\geq 0}$ is finite, aperiodic and irreducible for every $T>0$,
	and hence admits a unique steady state probability measure $\bm{\pi}_T$ on ${\cal V}$.
\end{lemma}

\begin{proof}
	Since $Y_T(t) \in {\cal V}$ for every $t \ge 0$ and $|{\cal V}| = 2^n$, 
	the DTMC clearly takes finitely many distinct values. To show irreducibility, we need to show that
	there exists positive probability path from any state $S$ to any other state $S'$ of the DTMC. This 
	follows by observing that from any DTMC state $S$ there exists a positive probability path to DTMC
	state $\phi$ and from $\phi$ there exists a positive probability path to any DTMC state $S'$.
	The DTMC is aperiodic since the DTMC state $\phi$ satisfies $P_{\phi\phi} > 0$. Thus,
	the DTMC is positive recurrent and has a unique stationary measure $\bm{\pi}_T$ that satisfies
	matrix fixed point equation $\bm{\pi}_T = \bm{\pi}_T \bm{P}$, where $\bm{P}$ is the TPM \cite{bre:13}.
\end{proof}

In the following result, we show that the steady state distribution has the desired form. 

\begin{theorem}
For any fixed $T>0$, the steady state distribution of the constructed DTMC  $\{{Y_{T}(t)}\}_{t\geq 0}$ is given by  
\begin{align}\label{eq:steadt_state_distr}
  \pi_T(S) = \frac{e^{-C(S)/T}}{\Sigma_{S_1\in{\cal V}}e^{-C(S_1)/T}}, \quad \forall S \in {\cal V}.
\end{align}
\end{theorem}
\begin{proof}
In order to prove the required, it is enough to show that the distribution $\bm{\pi}_T$ given in
Eq.~(\ref{eq:steadt_state_distr}) satisfies $\pi_{T}({S})P_T({S,S'}) =  \pi_{T}({S'})P_T({S',S})$ 
for every $S$, $S'\in {\cal V}$. 
This implies that the DTMC is time reversible and has the steady state distribution $\bm{\pi}_T$  \cite{bre:13}.
Recall that from $S$, the transitions can happen only to sets in ${\cal N}(S)$. Thus, for any
$S' \not\in {\cal N}(S)$, $P_T({S,S'})=P_T({S',S})=0$ and the required follows trivially.
When $S$ and $S'$ are neighbors, we consider the following two cases:

\

\noindent
$\bullet$ {\it Case 1:} $S' \in {\cal N}_1(S)$ -- Note that when $S' \in {\cal N}_1(S)$, 
then $S \in {\cal N}_2(S')$. Without loss of generality, let $C(S) > C(S')$. Then, using 
 Eqs.~(\ref{eq:trans1})  and (\ref{eq:steadt_state_distr}), it follows that
  \begin{align}
  \pi_{T}(S)P_T({S,S'}) &= \frac{1}{3n} \times\frac{e^{-C(S)/T}}{\Sigma_{S_1\in {\cal V} }e^{-C(S_1)/T}}, \nonumber \\
  &= \frac{e^{(C(S')-C(S))/T}}{3n}\times\frac{e^{-C(S')/T}}{\Sigma_{S_1\in {\cal V} }e^{-C(S_1)/T}}, \nonumber \\
  	&= \pi_{T}(S')P_T({S',S}). \nonumber
\end{align}
This proves the required. Also, note that $S' \in {\cal N}_2(S)$ follows using the exact same 
arguments as above.

\

\noindent
$\bullet$ {\it Case 2:} $S' \in {\cal N}_3(S)$ -- Note that when $S' \in {\cal N}_3(S)$, 
then $S \in {\cal N}_3(S')$. Moreover, $|S| = |S'|$. Without loss of generality, let $C(S) > C(S')$. Then, using  Eqs.~(\ref{eq:trans2})  and (\ref{eq:steadt_state_distr}), it follows that
\begin{align*}
\pi_T(S)P_T({S,S'}) &=  \frac{1}{3 |S| (n-|S|)} \times \frac{e^{-C(S)/T}}{\Sigma_{S_1\in {\cal V}}e^{-C(S_1)/T}}  \\
&= \frac{e^{(C(S')-C(S))/T}}{3 |S'| (n-|S'|)} \times \frac{e^{-C(S')/T}}{\Sigma_{S_1\in {\cal V} }e^{-C(S_1)/T}}, \\
&= \pi_T(S')P_T({S',S}).
\end{align*}

This proves the required.
\end{proof}

\begin{corollary}\label{cor:conv}
	For the DTMC $\{Y_T(t)\}_{t\geq0}$, following holds for every $S \in {\cal V}$:
	\[\proba(Y_T(t) = S) \to \frac{e^{-C(S)/T}}{\Sigma_{S_1\in{\cal V}}e^{-C(S_1)/T}} \mbox{ as $t\to\infty$.} \]
	Thus, 
	\begin{align}\label{eq:conv}
		\lim_{t\to\infty}\proba(Y_T(t) \in {\cal S}^\*) \to 1 \mbox{ as $T\to 0$,}
	\end{align}
	where \[ {\cal S}^\* = {\arg\min}_{ \{S\in {\cal V} : B(S) \in {\cal K}_A\}} |S|.\] 
\end{corollary}

The above corollary shows that if we consider the state of the DMTC for a large enough $t$ and a small enough $T$,
then it corresponds to an optimal solution with high probability. The parameter $T$ is referred to as ``temperature" and captures the classical
exploration versus exploitation trade-off. Large value of $T$ allows for the DTMC to explore more, but the
steady state probability mass on ${\cal S}^\*$ may not be high enough. Thus, the probability of finding optimal
solution may not be high enough. On the other hand, if $T$ is chosen to be small, then the exploration rate
becomes small and the chain may get stuck in the local minima for long time. We start with a large value of $T$
and then gradually reduce it. This allows rapid exploration in the beginning, and as time increases the chain
moves only to the states that have a lower cost. Let $T_t$ denote the temperature at time step $t$.
In our case, we start withe some predefined temperature $T_0$. Temperature parameter is updated after every $1000$ time steps. Our update rule is: $T_{1000 k} = \alpha T_{1000 (k-1)}$ with $\alpha < 1$ and $k=1,2,\ldots$. For $t \in \{1000 k,\ldots,1000(k+1) -1 \}$, $T_t = T_{1000 k}$. We conduct the simulations until
$T_t$ stays above a predefined value $T_{\rm stop}$. The value of the DTMC at the stopping time is the output of the algorithm.
Algorithm~\ref{alg:sa} provides the pseudo-code. 
\begin{lemma}
	Each iteration of Algorithm~\ref{alg:sa} has computational complexity $O(n^3)$. Moreover, the storage 
	required is $O(n^2)$. 
\end{lemma}
\begin{proof}
	Computationally most complex operation in an iteration of the algorithm is to compute $C(S_p)$, and it has complexity $O(n^3)$. Other operations are $O(n)$. Note that storing the input graph has space complexity $O(n^2)$. Other storage requirements are $O(n)$. This establishes the required.
\end{proof}

Next, we evaluate the performance of 
our algorithm.


\begin{algorithm}[t]
\caption{Pseudo-code for proposed MCMC
\label{alg:sa}}
\begin{algorithmic}
\State \textit {\bf Input:} Loop directed graphs $G(\A)$ and $G(\A_{\times}) $
\end{algorithmic}
\begin{algorithmic}
\State \textit{\bf Output:} A subset $S$ of $V$
\end{algorithmic}
\begin{algorithmic}[1]
\State Initialize the parameters ($T_0, T_{\rm stop}, \alpha$) \label{step:initialize}
\State $iter = 0$, $T \leftarrow T_0$ and $S = \phi$ 
\State $C(S) \leftarrow C(\phi)$
\While{$T_{\rm stop} \leq T$}
    \State Generate a sample $w$ from the probability mass function $[p_1(S),p_2(S),p_3(S)]$ 
    \If {$w = 1$ and $V\setminus S \not= \phi$} 
    \State Choose a vertex $x$ from $V\setminus S$ uniformly at random
    \State $S_p \leftarrow S \cup \{x\}$
    \ElsIf {$w = 2$ and $S \not= \phi$} 
    \State Choose a vertex $x$ from $V$ uniformly at random
    \State $S_p \leftarrow S \setminus \{x\}$
    \ElsIf {$w = 3$ and $S \not\in \{\phi,V\}$}
    \State Choose a vertex $x_1$ from $S$ and $x_2$ from $V\setminus S$ each uniformly at random
    \State $S_p \leftarrow (S \cup \{x_2\}) \setminus \{x_1\}$
    \Else 
    \State $S_p \leftarrow S$
    \EndIf
    \State Compute $C(S_p)$
    \State $S \leftarrow S_p$ and $C(S) \leftarrow C(S_p)$ w.p. $\min\{1,e^{(C(S) - C(S_p))/T}\}$ 
    \State $iter \leftarrow iter + 1$
    \If {$(iter \mod 1000) = 0 $}
      \State $T \gets T\alpha$ 
    \EndIf
\EndWhile
\end{algorithmic}
\end{algorithm}

\section{Results and Discussion}\label{sec:results}
Here, we present our experimental results. Performance of our algorithm is compared with optimal (whenever its computation is possible) and with only known heuristic \cite{ChaMes:13}. The following
scenarios are considered -- (1)~small random networks, (2)~large self-damped directed tree networks,
(3)~large random systems and (4)~real world networks.

In all our experiments, we use the following values for parameters in Algorithm~\ref{alg:sa}:
$T_0 = 1.5$, $\alpha = 0.95$ and $T_{\rm stop} = 0.001$. These choices amount to 143,000 iterations.
Next, we discuss our findings. 

\subsection{Small Random Graphs} \label{sec:small_graph}
We construct state matrix $\A$ with dimension varying from 5 to 100. 
Each entry in the matrix is $\times$ with probability $(1+\delta) \log(n)/n$ and zero otherwise,
where $\delta > 0$. 
Entries in the matrix are chosen independently. Note that $G(\A)$ for this construction is an
Erd\H{o}s-Renyi graph with the directed edge probability $(1+\delta) \log(n)/n$  \cite{erdos:60}.
Table~\ref{table:small} shows the results averaged over 100 random state matrices for each $n$.
Optimal results for $n> 20$ could not be obtained. Note that our randomized algorithm obtains
 optimal solution in every case and the heuristic algorithm does not. The results are somewhat expected
as for small $n$, the state space is not very large and it is possible that the randomized algorithm also
explores all the states before termination. Hence, next we consider large systems and see how the proposed
algorithm performs with respect to  optimal  solution and heuristic.

\begin{table} 
	
	\begin{center}
		\caption{Comparison of the number of inputs required for s-controllability under MCMC  and heuristic schemes. Results averaged over $100$ randomly generated networks}
		\label{table:small}
		\begin{tabular}{||c c c c||} 
			
			\hline
			Nodes & Optimal & MCMC & Proposed Heuristic\\ [0.5ex] 
			\hline\hline
			5 & 2.15 & 2.15 & 2.29  \\ 
			\hline
			8 & 2.86 & 2.86 & 3.28 \\ 
			\hline
			10  & 3.33 & 3.33 & 3.78 \\
			\hline
			12 & 3.56 & 3.56 & 4.16 \\  
			\hline
			15 & 4.37 & 4.37 & 5.05 \\ 
			\hline
			18 & 4.84 & 4.84 & 5.86\\
			\hline
			20 & 5.26 & 5.26 & 6.46\\
			\hline
			50 & NA & 11.58 & 14.45\\
			\hline 
			100 & NA & 23.45 & 28.63\\ 
			\hline
		\end{tabular}
	\end{center}

\end{table}

\subsection{Large Directed Tree Networks}\label{sec:tree_graph}
Here, we generate $\A$ matrix randomly such that $G(\A)$ is a directed tree with each node having
self loop ($\A_{ii} = \times$ for every $i$).
We consider these systems as solving Problem~\ref{prob:ssc} for tree networks is possible in polynomial time \cite{TreDel-15}. The results for network sizes 500, 1000, 1500 and 2000 are tabulated in the second column of Table~\ref{table:large}. Each entry in the column is calculated as $(b_{\rm H}(n)-b_{\rm R}(n))/b_{\rm R}(n) \times 100$, where $b_{\rm H}(n)$ and $b_{\rm R}(n)$ denote the average number of inputs chosen by the 
heuristic and our randomized algorithms, respectively, for the network of size $n$.
The average is over 100 different topologies for each network size. We like to mention that the proposed
algorithm has always achieved the optimal. Moreover, the heuristic's performance is about 150\% worse
than that of the proposed algorithm in each case. 

\begin{table} 

\begin{center}

\caption{Percentage reduction in the number of inputs required for s-controllability under the proposed scheme over that under the heuristic. Results averaged over $100$ randomly generated networks}
\label{table:large}
 \begin{tabular}{||c c c||} 

 \hline
 Nodes & Tree & Random System\\ [0.5ex] 
 \hline\hline
 500 & 149.44 & 12.69  \\ 
 \hline
 1000 & 162.52 & 7.54 \\ 
 \hline
 1500 &  164.53 & 5.4 \\
 \hline
 2000 & 170.80 & 3.36 \\  
 \hline
\end{tabular}
\end{center}
\end{table}

\subsection{Large Synthetic Random Graphs}
We conducted experiments on large synthetic random graphs for network sizes 500, 1000, 1500 and 2000. 
The networks are generated in the same way as described in Section~\ref{sec:small_graph}. 
The third column of Table~\ref{table:large} provides the performance comparison between the heuristic and 
the proposed algorithm. The performance metric is calculated in the same way as described in Section~\ref{sec:tree_graph}. In this case, the heuristic has performed reasonably well as compared to
the proposed algorithm. Moreover, the performance gap between the two schemes reduces as the network size 
increases. Here, we would like to note that though the considered ratio is decreasing, the absolute difference between the average number of inputs selected under the two schemes is increasing with $n$. 

\subsection{Real-world networks}
We further carried out experiments on real-world networks. The networks studied along with our findings are listed below:

\noindent
  $\bullet$ {\it Power system network}: controllability of power systems is necessary especially when studying for failures. We used IEEE 39-bus electric power system \cite{data:39bus} which is a benchmark model in power systems. The electric power grid consists of 39 bus system which act as vertices in the directed graph. These buses in the system are interconnected through transmission lines which act as edges. It is important to study which node (bus) needs to be applied with an input so that the electric grid remains controllable for all numerical realizations. Controllability of power systems allows to control the real and reactive powers in the system and to maintain the voltage and frequency of operation at the desired level.
  We applied the proposed algorithm and heuristic to minimize the number of buses to be selected. Results of both the methods are listed in Table~\ref{table:res_real}.
  
  \noindent
  $\bullet$ {\it Social network}: social networks play a crucial role in spread of information. Controllability in social networks aims at modifying the dynamics of a social network in order to drive the states of the network to a set of states that maximizes the welfare of the network. A popular social network is Twitter in which users on the platform can make posts referred as {\it tweets}. We have considered {\it Re-tweet Network} from the UN conference held in Copenhagen \cite{data:twitter}. The network contains 761 nodes each representing a user on Twitter. Edges represent a {\it re-tweet} or {\it mention}. A re-tweet occurs when a user posts the tweet of another user. A mention occurs if a user refers to another user in the post. We minimize the number of users to whom inputs are directly applied in order to be able to control the information flow in the network. Results of the proposed scheme and the heuristic are in Table~\ref{table:res_real}.
  
  \noindent
  $\bullet$  {\it Biological network}: controllability of biological networks has been studied to identify drug targets, protein interaction etc. Modifications and interactions between proteins affect the human health. We use the dataset of Stelzl protein interaction in humans \cite{data:stelzl}. Each vertex represent a protein and edges represent the interaction between them. 
   Results of the proposed algorithm and the heuristic on the biological network are in Table \ref{table:res_real}.
  
  \noindent
  $\bullet$ {\it Infrastructure}: networks can also be constructed for several infrastructure scenarios such as road network and airport network. Controllability of these networks is important for operational reasons. 
  In an airport network, the state nodes correspond to the airports and controllability guarantees a congestion-free operation of the network. We consider the network of flights between US airports in the year 2010 \cite{data:USairport}. Each node in the network represents an airport, and an edge represents a flight connecting them.   Results for s-controllability of these networks using proposed algorithm and the heuristic are presented in Table \ref{table:res_real}.
  
  \noindent
  $\bullet$ { \it Interaction}: flow of ideas and information takes place due to one-to-one interaction among individuals. Every person belongs to different interaction networks and studying the controllability of such networks helps in analyzing the information flow. We study a network that records innovation spread among $246$ physicians belonging to Illinois, Peoria, Bloomington, Quincy, and Galesburg \cite{data:physicians}. Each node represents a physician while an edge represents a physician's interaction to other physician for advice/discussion or unofficial friendly talk. Results for the network are presented in table \ref{table:res_real}.

Note that for sparse networks the heuristic and the proposed randomized algorithm has similar performance.
However, when the graphs become dense on account of increased interactions between network entities, then 
the proposed algorithm significantly outperforms the heuristic.

\begin{table} 
\caption{The number of inputs required for s-controllability of 5 real-world networks using MCMC and heuristic}
\begin{center}
 \begin{tabular}{||c c c c c||} 
\hline 
 Network & Nodes & Edges & MCMC & Heuristic\\ [0.5ex] 
 \hline\hline
 IEEE 39-bus & 39 & 46 & 14 & 14  \\ 
 \hline
 Retweet  & 761 & 1029 & 466 & 466\\ 
 \hline
 \parbox[c]{2.2cm}{Protein-protein \\ interaction}  & 1706 & 6207 & 815 & 819 \\
 \hline
 US airport  & 1574 & 28236 &672 & 715 \\
 \hline
 Physicians & 241 & 1098 & 54 & 76 \\
 \hline
\end{tabular}
\end{center}

\label{table:res_real}
\end{table}
\section{Conclusion}\label{sec:conclu}
Controllability of LTI systems for all numerical realizations (s-controllability) is desired in several scenarios. 
We have considered a problem of finding minimum number of inputs so as to achieve s-controllability.
Unfortunately, this problem is not only NP-hard, but it is also hard to approximate. Indeed, the only known heuristic
for this problem can be shown to perform arbitrarily bad. Hence, we proposed a randomized algorithm that outputs an
optimal solution with high probability. We have tested the proposed algorithm on  large systems for which optimal solution can be computed, and found that the algorithm finds an optimal solution in reasonable time. Even for other large randomly
generated systems the proposed scheme outperforms the heuristic. Finally, we show that the performance of the
proposed scheme is significantly better than the heuristic in several real world networks of practical importance.

\bibliographystyle{myIEEEtran}      
\bibliography{myreferences}
\appendices
\section{Illustrative Example} \label{app:illustrative_ex}
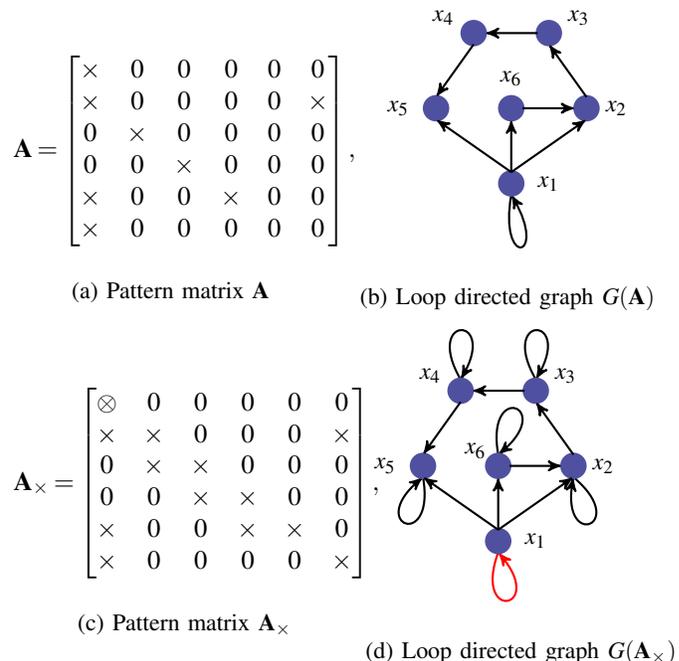
\begin{figure}[h]
	\centering
		\begin{subfigure}{.23\textwidth}
		\vspace{0.4cm}
		\begin{equation*} \label{eq:Amatrix}
		\A = 
		\begin{bmatrix}
		\times & 0 & 0 & 0 & 0 & 0\\
		\times & 0 & 0 & 0 &0 & \times\\
		0 & \times & 0 & 0 & 0 & 0 \\
		0 & 0 & \times & 0 & 0 & 0 \\
		\times & 0 & 0 & \times & 0 & 0 \\
		\times & 0 & 0 & 0 & 0 & 0\\
		\end{bmatrix},
		\end{equation*}
		\caption{ Pattern matrix $\A$} \label{fig:eg1}
	\end{subfigure}
	\begin{subfigure}{.25\textwidth}
		
		\centering
		\begin{tikzpicture}[scale = 1, ->,>=stealth',shorten >=0.5pt,auto,node distance=2cm,
		thick,main node/.style={circle,draw,font=\Large\bfseries}]
		\tikzset{every loop/.style={min distance=10mm,looseness=10}}
		
		\path[->] (0,-2.15) edge [in=-60,out=-100,loop] node[auto] {} ();
		
		\definecolor{myblue}{RGB}{80,80,160}
		\definecolor{mygreen}{RGB}{80,160,80}
		\definecolor{myred}{RGB}{144, 12, 63}
		\definecolor{myyellow}{RGB}{214, 137, 16}
		
		\node at (-0.9,0.2) {\small $x_4$};
		\node at (1.4,-1) {\small $x_2$};
		\node at (0.5,-2) {\small $x_1$};
		\node at (-1.5,-1) {\small $x_5$};
		\node at (0.9,0.2) {\small $x_3$};
		\node at (0,-0.6) {\small $x_6$};
		
		\fill[myblue] (-0.5,0) circle (5.0 pt);  
		\fill[myblue] (1.0,-1) circle (5.0 pt); 
		\fill[myblue] (0,-2) circle (5.0 pt); 
		\fill[myblue] (-1.0,-1) circle (5.0 pt);
		\fill[myblue] (0.5,0) circle (5.0 pt); 
		\fill[myblue] (0,-1) circle (5.0 pt); 
		
		\draw (-0.5,-0.15)  ->  (-1,-0.85); 
		\draw (0,-1.85)  ->  (-1.0,-1.15);
		\draw (0,-1.85)  ->  (1.0,-1.15); 
		\draw (1.0,-0.85)  ->  (0.5,-0.15); 
		\draw (0.35,0)  ->  (-0.35,0); 
		\draw (0,-1.85)  ->  (0,-1.15); 
		\draw (0.15,-1)  ->  (0.85,-1); 
		\end{tikzpicture}
		\caption{Loop directed graph $G(\A)$}
		\label{fig:digraph}
	\end{subfigure}%
\hspace{0.5 cm}
	\begin{subfigure}{.25\textwidth}
	\vspace{0.4cm}
	\begin{equation*} \label{eq:Axmatrix}
	\A_{\times} = 
	\begin{bmatrix}
	\otimes & 0 & 0 & 0 & 0 & 0\\
	\times & \times & 0 & 0 &0 & \times\\
	0 & \times & \times & 0 & 0 & 0 \\
	0 & 0 & \times & \times & 0 & 0 \\
	\times & 0 & 0 & \times & \times & 0 \\
	\times & 0 & 0 & 0 & 0 & \times\\
	\end{bmatrix},
	\end{equation*}
	\caption{Pattern matrix $\A_{\times}$}\label{fig:eg2}
\end{subfigure}
	\begin{subfigure}{.23\textwidth}
		\begin{tikzpicture}
		[scale = 1, ->,>=stealth',shorten >=0.5pt,auto,node distance=2cm,
		thick,main node/.style={circle,draw,font=\Large\bfseries}]
		\tikzset{every loop/.style={min distance=10mm,looseness=10}}

		\path[->] (0,-0.85) edge [in=40,out=100,loop] node[auto] {} ();
		\path[->] (0.5,0.15) edge [in=60,out=120,loop] node[auto] {} ();
		\path[->] (-1.0,-1.15) edge [in=220,out=280,loop] node[auto] {} ();
		\path[->] [red] (0,-2.15) edge [in=-50,out=-110,loop] node[auto] {} ();
		\path[->] (1.0,-1.15) edge [in=-100,out=-40,loop] node[auto] {} ();
		\path[->] (-0.5,0.15) edge [in=60,out=120,loop] node[auto] {} ();
		
		\definecolor{myblue}{RGB}{80,80,160}
		\definecolor{mygreen}{RGB}{80,160,80}
		\definecolor{myred}{RGB}{144, 12, 63}
		\definecolor{myyellow}{RGB}{214, 137, 16}
		
		\node at (-0.9,0.2) {\small $x_4$};
		\node at (1.4,-1) {\small $x_2$};
		\node at (0.5,-2) {\small $x_1$};
		\node at (-1.5,-1) {\small $x_5$};
		\node at (0.9,0.2) {\small $x_3$};
		\node at (-0.3,-0.8) {\small $x_6$};
		
		\fill[myblue] (-0.5,0) circle (5.0 pt);  
		\fill[myblue] (1.0,-1) circle (5.0 pt); 
		\fill[myblue] (0,-2) circle (5.0 pt); 
		\fill[myblue] (-1.0,-1) circle (5.0 pt);
		\fill[myblue] (0.5,0) circle (5.0 pt); 
		\fill[myblue] (0,-1) circle (5.0 pt); 
		
		\draw (-0.5,-0.15)  ->  (-1,-0.85); 
		\draw (0,-1.85)  ->  (-1.0,-1.15);
		\draw (0,-1.85)  ->  (1.0,-1.15); 
		\draw (1.0,-0.85)  ->  (0.5,-0.15); 
		\draw (0.35,0)  ->  (-0.35,0); 
		\draw (0,-1.85)  ->  (0,-1.15); 
		\draw (0.15,-1)  ->  (0.85,-1); 
		\end{tikzpicture}
		\caption{Loop directed graph $G(\A_{\times})$}
		\label{fig:digraph_Ax}
	\end{subfigure}
\caption{\small State matrix $\A$ with corresponding $\A_{\times}$ and the graphs $G(\A)$ and $G(\A_\times)$}\label{fig:eg}

\end{figure}

Figure~\ref{fig:eg1} shows the structural state matrix $\A$. 
Corresponding $G(\A)$, $\A_{\times}$ and $G(\A_\times)$ are shown in Figures~\ref{fig:digraph}, \ref{fig:eg2} and
\ref{fig:digraph_Ax}, respectively (recall Section~\ref{sec:prelim}). Consider loop directed graph $G(\A)$. Initially, consider $S = \phi$.
Note that $x_5$ is a lone out-neighbor of $x_4$, and hence as per colour
changing rule (see Definition~\ref{defi:ccr}), $x_4$ can force $x_5$ to turn
black. A chronological list of forces in $G(\A)$ is as follows:
(1)~$x_4\rightarrow x_5$, (2)~$x_3 \to x_4$, (3)~$x_2\to x_3$ and 
(4)~$x_6\to x_2$. At this point, nodes $x_2, x_3, x_4$ and $x_5$ 
have become black, but no force is possible to change the colours of $x_1$
and $x_6$ as both these nodes are out-neighbor of the same node $x_1$. Thus,
$\phi$ is not a ZFS for $G(\A)$, but either of $\{x_1\}$ and 
$\{x_6\}$ is a ZFS for $G(\A)$. Note that if $x_6$ ($x_1$, resp.) is chosen to be
black initially, then after applying the above forces $x_1$ can force itself ($x_6$, resp.).
Next consider $G(\A_\times)$
and let $S = \{x_6\}$ which is a zero forcing set of $G(\A)$. Following is a chronological sequence of all possible forces 
($i$)~$x_5\to x_5$, ($ii$)~$x_4\to x_4$, 
($iii$)~$x_3\to x_3$, and ($iv$)~$x_2\to x_2$. At this point all nodes except $x_1$ are black. Note 
that though $x_1$ remains to be only white out-neighbor of itself, force
$x_1 \to x_1$ is not allowed (see condition~2) in Proposition~\ref{prop:ssc_thm}).
Thus, $S=\{x_6\}$ does not satisfy both the requirements in Proposition~\ref{prop:ssc_thm}. However,
$S=\{x_1\}$ satisfies the requirement and hence giving input at $x_1$ makes the system  s-controllable. 
Equivalently, $(\A,\B(\{x_1\}))$ is s-controllable.

\section{Performance Analysis of a Heuristic Algorithm for Strong Structural Controllability}
The following example shows that the heuristic proposed by  Chapman {\it et al.} \cite{ChaMes:13} for minimizing the input cardinality required for strong structural controllability can given $O(n)$ suboptimal result. Note that when dedicated inputs are considered, the sparsest input design problem is equivalent to minimum input cardinality selection problem. Consider the following state matrix $\bf A$ and modified matrix $\bf A_\times$

\setcounter{MaxMatrixCols}{15}
\setlength\arraycolsep{3pt}
\[
\bf A=
\begin{bmatrix}
0& 0&0&0& \times & 0 &\cdots & \cdots & \cdots & \cdots & \cdots & \cdots & \cdots& \cdots& 0 \\
0& 0&0&0&\times &  \times & 0 & \cdots & \cdots &\cdots &\cdots &\cdots & \cdots & \cdots& 0 \\
0& 0&0&0& \times &  \times & \times & 0 & \cdots & \cdots &\cdots & \cdots & \cdots& \cdots& 0 \\
0& 0&0&0& 0 &  \times & \times & \times & 0 & \cdots &\cdots & \cdots & \cdots& \cdots& 0 \\
0& 0 &0&0&0&  \times & \times & \times & \times & 0 &\cdots &\cdots &\cdots &\cdots &0\\
0&   0 &0&0&0&  \times & \times & \times & \times & \times &0 &\cdots &\cdots &\cdots &0\\
0&0&0& 0&  0 &  \times & \times & \times & \times & \times &\times &0 &\cdots &\cdots &0\\
0&0&0&0&   0 & 0 & 0 & 0 & \times & \times &\times &\times &0 &\cdots&0\\
0&0&0& 0&  0 & 0 & 0 & 0 & \times & \times &\times &\times &\times &0 &0\\
0&0&0&  0&   0 & 0 & 0 & 0 & \times & \times &\times &\times &\times &\times &0 \\
0&0&0& 0&   0 & 0 & 0 & 0 & \times & \times &\times &\times &\times &\times &\times \\
0&0&0&   0&   0 & 0 & 0 & 0 & 0 & 0 &0 &\times &\times &\times &\times  \\
0&0&0&   0&   0 & 0 & 0 & 0 & 0 & 0 &0 &\times &\times &\times &\times   \\
0&0&0&  0&   0 & 0 & 0 & 0 & 0 & 0 &0 &\times &\times &\times &\times   \\
0&0&0&    0&   0 & 0 & 0 & 0 & 0 & 0 &0 &\times &\times &\times &\times  \\
\end{bmatrix}
\] 

\setcounter{MaxMatrixCols}{15}
\setlength\arraycolsep{3pt}
\[
\bf A_\times=
\begin{bmatrix}

\otimes& 0&0&0& \times & 0 &\cdots & \cdots & \cdots & \cdots & \cdots & \cdots & \cdots& \cdots& 0 \\
0& \otimes&0&0&\times &  \times & 0 & \cdots & \cdots &\cdots &\cdots &\cdots & \cdots & \cdots& 0 \\
0& 0&\otimes&0& \times &  \times & \times & 0 & \cdots & \cdots &\cdots & \cdots & \cdots& \cdots& 0 \\
0& 0&0&\otimes& 0 &  \times & \times & \times & 0 & \cdots &\cdots & \cdots & \cdots& \cdots& 0 \\
0& 0 &0&0&\otimes&  \times & \times & \times & \times & 0 &\cdots &\cdots &\cdots &\cdots &0\\
0&   0 &0&0&0&  \otimes & \times & \times & \times & \times &0 &\cdots &\cdots &\cdots &0\\
0&0&0& 0&   0 &  \times & \otimes & \times & \times & \times &\times &0 &\cdots &\cdots &0\\
0&0&0&0&   0 & 0 & 0 & \times & \times & \times &\times &\times &0 &\cdots&0\\
0&0&0& 0&  0 & 0 & 0 & 0 & \otimes & \times &\times &\times &\times &0 &0\\
0&0&0&  0&   0 & 0 & 0 & 0 & \times & \otimes &\times &\times &\times &\times &0 \\
0&0&0& 0&   0 & 0 & 0 & 0 & \times & \times &\otimes &\times &\times &\times &\times \\
0&0&0&   0&   0 & 0 & 0 & 0 & 0 & 0 &0 &\otimes &\times &\times &\times  \\
0&0&0&   0&   0 & 0 & 0 & 0 & 0 & 0 &0 &\times &\otimes &\times &\times   \\
0&0&0&  0&   0 & 0 & 0 & 0 & 0 & 0 &0 &\times &\times &\otimes &\times   \\
0&0&0&    0&   0 & 0 & 0 & 0 & 0 & 0 & 0 &\times &\times &\times &\otimes  \\
\end{bmatrix}
\] 

Assume the state matrix $\bf A$ to be of dimension $n\times n$ with the same pattern as described. The heuristic proposed in \cite{ChaMes:13} is first applied on state matrix $\bf A$ to obtain set $S_1$. This set $S_1$ has cardinality $(3(n-3)/4) +2$. Using obtained set $S_1$, the input matrix ${\bf B}(S1)$ is formed and the heuristic is now applied on the pair  $ ({\bf A},{\bf B}(S_1))$ returning set $S_2$. Total required inputs is given by the set $S_1 \cup S_2$. Since $S_1$ has cardinality of $O(n)$, the cardinality of $S_1 \cup S_2$ is also $O(n)$. However, the following input matrix $\bf B$ of size $n \times 4$ makes the system strong structurally controllable.

\setcounter{MaxMatrixCols}{15}
\setlength\arraycolsep{3pt}
\[
\bf B=
\begin{bmatrix}
0 & 0 & 0 & 0 \\
\vdots & \vdots & \vdots & \vdots \\
\vdots & \vdots & \vdots & \vdots \\
\vdots & \vdots & \vdots & \vdots \\
0 & 0 & 0 & 0 \\
\times & 0 & 0 & 0 \\
0 & \times & 0 & 0 \\
0 & 0 & \times & 0 \\
0 & 0 & 0 & \times \\

\end{bmatrix}
\]  

Thus, in this example the input design matrix requires input matrix $\bf B$ of cardinality $4$, but the heuristic proposed in \cite{ChaMes:13} gives the cardinality as $O(n)$.
\end{document}